\documentclass[11pt,reqno]{amsart}
\usepackage[a4paper,margin=30mm]{geometry}
\usepackage{amsmath,amssymb,amsthm,mathtools}
\usepackage{color}
\usepackage{booktabs,microtype}
\usepackage[colorlinks=true,linkcolor=blue,citecolor=blue,urlcolor=blue]{hyperref}
\usepackage{xcolor}

\newtheorem{theorem}{Theorem}[section]
\newtheorem{conjecture}[theorem]{Conjecture}
\newtheorem{proposition}[theorem]{Proposition}

\newtheorem{remark}[theorem]{Remark}
\numberwithin{equation}{section}

\newcommand{\dd}{\,\mathrm d}

\begin{document}
	
\title[Counterexamples to Escobar's conjecture]{Counterexamples to Escobar's conjecture}

\author{Jin Sun}
\address{Jin Sun: School of Mathematical Sciences, Fudan University, 200433, Shanghai, China}
\email{\href{mailto:jsun22@m.fudan.edu.cn}{jsun22@m.fudan.edu.cn}}
	
\author{Lili Wang}
\address{Lili Wang: School of Mathematics and Statistics, Key Laboratory of Analytical Mathematics and Applications (Ministry of Education), Fujian Key Laboratory of Analytical Mathematics and Applications (FJKLAMA), Fujian Normal University, 350117, Fuzhou, China}
\email{\href{mailto:liliwang@fjnu.edu.cn}{liliwang@fjnu.edu.cn}}
	
\author{Tao Wang}
\address{Tao Wang: Beijing International Center for Mathematical Research, Peking University, 100871, Beijing, China}
\email{\href{mailto:taowang25@pku.edu.cn}{taowang25@pku.edu.cn}}

\date{\today}	
	
\begin{abstract}
Escobar (J Funct Anal 165(1):101-116, 1999) conjectured that for every $n \geq 3$, an $n$-dimensional compact Riemannian manifold with nonnegative Ricci curvature and all boundary principal curvatures bounded below by $\kappa>0$ must satisfy $\sigma_1\geq \kappa$. We disprove this conjecture for every $n\geq 100$ by constructing conformal deformations of the Euclidean unit ball. We first establish a perturbative criterion, then construct explicit polynomial conformal factors satisfying this criterion. For every sufficiently small $t>0$, the resulting metrics $g_t=e^{2t\Phi}g_{\mathbb{R}^n}$ have positive Ricci curvature, every boundary principal curvature is strictly larger than $1$, and $\sigma_1(\mathbb{B}^n,g_t)<1$.

\end{abstract}
	
\maketitle
	
\section{Introduction}\label{sec:introduction}
	
Let $(M^n,g)$ be an $n$-dimensional smooth compact connected Riemannian manifold with nonempty boundary $\Sigma=\partial M$, and let $\nu$ be the outward unit normal. The Steklov problem, introduced by Stekloff \cite{Stekloff1902}, is
\[
 \Delta_g u=0\quad\text{ in }M,
 \qquad
 \partial_\nu u=\sigma u\quad\text{ on }\Sigma.
\]
Its spectrum is discrete and can be written as $0=\sigma_0<\sigma_1\leq\sigma_2\leq\cdots\nearrow\infty$. The first nonzero eigenvalue has the variational characterization
\begin{equation}\label{eq:rayleigh}
{\sigma_1(M,g)=
	\inf_{\substack{u\in H^1(M)\\ \int_\Sigma u\dd S_g=0\\ \int_\Sigma u^2\dd S_g=1}}{\displaystyle\int_M|\nabla_g u|_g^2\dd V_g}.}
\end{equation}
The Steklov spectrum is the spectrum of the Dirichlet-to-Neumann operator and therefore reflects both interior and boundary geometry. For a survey of Steklov eigenvalues and related results, we refer to \cite{CGGS2024,GirouardPolterovich2017}. 
	
Sharp upper bounds for $\sigma_1$ have a long history. Weinstock \cite{Weinstock1954} proved that the disk maximizes $\sigma_1$ among simply connected planar domains with fixed perimeter. In higher dimensions, Brock \cite{Brock2001} proved that the ball maximizes $\sigma_1$ among Euclidean domains of fixed volume. For convex Euclidean domains, Bucur--Ferone--Nitsch--Trombetti \cite{BFNT2021} obtained the sharper version in which the boundary area constraint replaces the volume constraint. Fraser and Schoen \cite{FraserSchoen2011, FraserSchoen2019} developed conformal and minimal surface methods in Steklov eigenvalue optimization.

Lower bounds are usually established under curvature assumptions. Payne \cite{Payne1970} showed that if $\Omega\subset\mathbb{R}^2$ is a smooth bounded convex domain whose boundary curvature is at least $\kappa>0$, then $\sigma_1(\Omega)\geq \kappa$, with equality only for the disk of radius $1/\kappa$. Escobar \cite{Escobar1997} extended this estimate to compact Riemannian surfaces with nonnegative Gaussian curvature and boundary geodesic curvature bounded below by $\kappa$. The maximum principle arguments underlying these results are specific to dimension two. Throughout, for a compact Riemannian manifold with boundary, we denote by $\mathrm{II}_g$ the second fundamental form of the
boundary, defined by
\[
 \mathrm{II}_g(X,Y)=g(\nabla_X\nu,Y),\qquad X,Y\in T\Sigma,
\]
where $\nu$ is the outward unit normal. In higher dimensions, Escobar \cite{Escobar1997} used Reilly's formula to prove that if $n\ge3$, $\mathrm{Ric}_g\ge0$, and $\mathrm{II}_g\ge\kappa\,g|_{\Sigma}$ for some $\kappa>0$, then $\sigma_1>\kappa/2$. He subsequently proposed the following conjecture.

\begin{conjecture}[\cite{Escobar1999}]\label{conj:Escobar}
Let $(M^n,g)$, $n\geq3$, be a smooth compact connected Riemannian manifold with nonempty boundary $\Sigma$. Suppose that
\[
 \mathrm{Ric}_g\geq0\quad\text{on } M,\qquad \mathrm{II}_g\geq\kappa g|_\Sigma\quad\text{on } \Sigma
\]
for some constant $\kappa>0$. Then
\[
\sigma_1(M,g)\geq\kappa.
\]
Moreover, equality holds if and only if $(M,g)$ is isometric to the Euclidean ball of radius $1/\kappa$.
\end{conjecture}

Xia and Xiong \cite{XiaXiong2024} proved Conjecture~\ref{conj:Escobar} under the stronger assumption of nonnegative sectional curvature. Under the original Ricci and boundary principal curvatures hypotheses, Duncan and Kumar \cite{DuncanKumar2025} improved lower bounds. {Under the additional assumption $\mathrm{Sec}_g\ge -K$ for some $K>0$, Liu and Yang~\cite{LiuYang2026-cam} proved the lower bound $\sigma_1 > \kappa - K/\kappa $}. Later, Liu and Yang \cite{LiuYang2026} introduced a weight function depending on the radial sectional curvature lower bound. This weight yields a unified lower bound for $\sigma_1$ that interpolates between Escobar's estimate $\sigma_1>\kappa/2$ and the sharp bound $\sigma_1\ge\kappa$ under nonnegative sectional curvature. 

However, these results leave open whether Conjecture~\ref{conj:Escobar} holds under its original assumptions alone.  In this paper, we construct counterexamples for arbitrarily high dimensions to show that the conjecture fails. We refer to another independent work \cite{AT2026} for a different construction of counterexamples in the case $n=3$.

Let $\mathbb{B}^n \subset \mathbb{R}^n$ be the unit ball, $\mathbb{S}^{n-1} = \partial \mathbb{B}^n$, $\omega_n = |\mathbb{B}^n|$, and $\delta = g_{\mathbb{R}^n}$. For a smooth function $\Phi$ on $\overline{\mathbb{B}^n}$, we consider the conformal metric
\[
 g_t=e^{2t\Phi}\delta.
\]
Define
 \begin{equation}\label{eq:linearized-data}
  \mathcal{A}_\Phi:=-(n-2)\mathrm{Hess}\,\Phi-(\Delta\Phi)\delta,
		\qquad
  q_\Phi:=\partial_\nu\Phi-\Phi,
 \end{equation}
and
 \begin{equation}\label{eq:linearized-rayleigh}
  \mathcal{L}_n(\Phi):=\frac1{\omega_n}\left((n-2)\int_{\mathbb{B}^n}\Phi\dd x -(n-1)\int_{\mathbb{S}^{n-1}}x_1^2\Phi \dd S\right).
 \end{equation}
Here, {$\mathcal{A}_{\Phi}$ is the first variation of the Ricci tensor, $q_{\Phi}$ is the first variation of the boundary principal curvatures and $\mathcal{L}_n$ is the derivative of the Rayleigh quotient of the test function $x_1$.} 
{Throughout the perturbative calculation, $\nabla$, $\mathrm{Hess}$, and $\Delta$ are taken with respect to the Euclidean metric $\delta$.}
The following criterion is the perturbative core of the construction. 
\begin{theorem}\label{thm:criterion}
 Let $n\ge3$, and let $\Phi$ be a smooth even function on $\overline{\mathbb{B}^n}$ with respect to $x_1$, i.e., 
\[
\Phi(-x_1,x')=\Phi(x_1,x'), \qquad {x=(x_1,x')\in \overline{\mathbb{B}^{n}}}.
\] 
Suppose that
 \begin{equation}\label{eq:three-signs}
  \mathcal{A}_\Phi>0\quad\text{on }\ \overline{\mathbb{B}^n},
		\qquad
  q_\Phi>0\quad\text{on } \ \mathbb{S}^{n-1},
		\qquad
  \mathcal{L}_n(\Phi)<0.
\end{equation}
Then there exists $t_0>0$ such that, for every $0<t<t_0$, the conformal metric
\[
 g_t=e^{2t\Phi}\delta
\]
satisfies
\[
 \mathrm{Ric}_{g_t}>0, \qquad \mathrm{II}_{g_t}>g_t|_{\partial\mathbb{B}^n}, \qquad \sigma_1(\mathbb{B}^n,g_t)<1.
\]
Moreover, {the set of functions satisfying \eqref{eq:three-signs} is open in the space of smooth functions that are even in $x_1$, equipped with the $C^2$ topology}. Thus, each such $\Phi$ yields an infinite-dimensional open family of counterexamples.
\end{theorem}

Our second theorem supplies a polynomial satisfying all three strict inequalities in every dimension $n \geq 100$.

\begin{theorem}\label{thm:polynomial}
For every integer $n\geq 100$, there exists {a smooth even function $\Phi_n$ on $\overline{\mathbb{B}^n}$ with respect to $x_1$,} such that
\[
\mathcal{A}_{\Phi_n}>0\quad\text{on } \ \overline{\mathbb{B}^n}, \qquad q_{\Phi_n}>0\quad\text{on } \ \mathbb{S}^{n-1}, \qquad \mathcal{L}_n(\Phi_n)<0.
\]
Consequently, if $g_t=e^{2t\Phi_n}\delta$, then the Rayleigh quotient of the test function $x_1$ is less than $1$ for every sufficiently small $t>0$. In particular,
\[
\mathrm{Ric}_{g_t}>0, \qquad \mathrm{II}_{g_t}>g_t|_{\partial\mathbb{B}^n},\qquad \sigma_1(\mathbb{B}^n,g_t)<1.
 \]
Therefore, Conjecture~\ref{conj:Escobar} fails in every dimension $n\geq 100$.
\end{theorem}

This paper is organized as follows. In Section~\ref{sec:criterion}, we prove Theorem~\ref{thm:criterion}. Section~\ref{sec:quadratic} explores the quadratic polynomial, identifies a sufficient feasible region, and constructs counterexamples for $n\geq100$. 
\section{Proof of the perturbative criterion}\label{sec:criterion}

In this section, we prove the perturbative criterion (Theorem \ref{thm:criterion}). The perturbative argument combines the standard conformal transformation formulas with the classical use of coordinate functions as test functions in the Steklov variational characterization \eqref{eq:rayleigh}.
\begin{proof}[Proof of Theorem \ref{thm:criterion}]
For the conformal transformation $\widetilde g=e^{2f}\delta$, we have
\[
 \mathrm{Ric}_{\widetilde g} =-(n-2)\mathrm{Hess}\, f-(\Delta f)\delta +(n-2)\left(\dd f\otimes\dd f-|\nabla f|^2\delta\right).
\]
Therefore, when $f=t\Phi$, we have
\begin{equation}\label{eq:ricci-expansion}
 \mathrm{Ric}_{g_t} =t\mathcal{A}_\Phi +(n-2)t^2\left(\dd\Phi\otimes\dd\Phi-|\nabla\Phi|^2\delta\right).
\end{equation}
Since $\Phi$ is a smooth function on $\overline{\mathbb{B}^n}$, $|\nabla \Phi|^2$ is bounded on $\overline{\mathbb{B}^n}$. Hence, from $\mathcal{A}_\Phi > 0$ on $\overline{\mathbb{B}^n}$, we obtain $\mathrm{Ric}_{g_t} > 0$ for any sufficiently small $t>0$.

Note that {$\mathrm{II}_{\delta} = \delta$}. By the standard conformal change transformation for boundary principal curvatures (see, e.g., \cite{Escobar1992}), for each $i=1,\dots,n-1$, we have 
\begin{equation}\label{eq:kappa-expansion}
\kappa_i(t)=e^{-t\Phi}\left(1+t\partial_\nu\Phi\right) =1+tq_\Phi+O(t^2),
\end{equation}  
uniformly on $\mathbb{S}^{n-1}$, where $\nu$ is the outward unit normal vector. Therefore $q_\Phi>0$ implies $\kappa_i(t)>1$ for small positive $t$. That is, $\mathrm{II}_{g_t}>g_t|_{\partial\mathbb{B}^n}$.
 
Since $\Phi$ is an even function with respect to $x_1$, we have
\[
 \int_{\mathbb{S}^{n-1}}x_1 \dd S_{g_t} =\int_{\mathbb{S}^{n-1}}x_1 e^{(n-1)t\Phi} \dd S=0.
\]
Therefore, the coordinate $x_1$ can serve as a test function in the variational characterization \eqref{eq:rayleigh}. Under the conformal metric $g_t=e^{2t\Phi}\delta$, the volume element is $e^{nt\Phi}\dd x$, and the boundary area element is $\dd S_{g_t}=e^{(n-1)t\Phi}\dd S$. Combining this with $|\nabla_{g_t}x_1|^2_{g_t}=e^{-2t\Phi}|\nabla x_1|^2=e^{-2t\Phi}$, we obtain that the Rayleigh quotient of the test function $x_1$ is 
\begin{equation}\label{eq:R-general}
R(t)=\frac{\displaystyle\int_{\mathbb{B}^n}e^{(n-2)t\Phi} \dd x}{\displaystyle\int_{\mathbb{S}^{n-1}}x_1^2e^{(n-1)t\Phi}\dd S}.
 \end{equation}
At $t=0$, since
\[
\int_{\mathbb{B}^n}1 \dd x=\omega_n 
\]
and 
\[
 \int_{\mathbb{S}^{n-1}}x_1^2 \dd S=\frac{1}{n}|\mathbb{S}^{n-1}|=\omega_n,
\]
we have $R(0)=1$. Therefore, we obtain
\[
 R'(0) = \frac{1}{\omega_n}\left((n-2)\int_{\mathbb{B}^n}\Phi \dd x -(n-1)\int_{\mathbb{S}^{n-1}}x_1^2\Phi \dd S \right) = \mathcal{L}_n(\Phi)<0.
\]
Hence, for sufficiently small $t>0$,  we have $R(t)<1$. By the variational characterization \eqref{eq:rayleigh}, we obtain $\sigma_1(\mathbb{B}^n,g_t)\le R(t)<1$.

Since the inequalities in \eqref{eq:three-signs} are all strict, they remain valid {under sufficiently small} $C^2$ perturbations of $\Phi$. Therefore, {the set of functions satisfying \eqref{eq:three-signs} is open in the space of smooth functions that are even in $x_1$, in the $C^2$ topology}.
\end{proof}

\section{\texorpdfstring{Quadratic polynomial counterexamples in high dimension}{Quadratic polynomial counterexamples in high dimension}}\label{sec:quadratic}

We now specialize Theorem~\ref{thm:criterion} to quadratic $O(n-1)$-invariant polynomials
\[
\Phi(x)=P(a,b),\qquad
a=x_1^2,\qquad b=x_2^2+\cdots+x_n^2,
\]
where $P(a,b)$ is a quadratic polynomial in $a$ and $b$. Thus, $(a,b)$ lies in the triangle  
\[
\Delta_2:=\{(a,b):a\geq0,\ b\geq0,\ a+b\leq1\}.
\]
Set $\rho=1-a-b$, so that $(\rho,a,b)$ are barycentric coordinates on $\Delta_2$. For an affine function $f(a, b) = c + p_1a + p_2b$ defined on $\Delta_2$, we can rewrite it as
\[
 f(a, b) = c(\rho + a + b) + p_1a+p_2b = c\rho + (p_1 + c)a + (p_2+c)b.
\]
So if $c, p_1 +c, p_2 +c > 0$, then $f(a, b) > 0$ on $\Delta_2$.
 
Direct calculation shows that linear polynomials are incompatible with the three sign conditions in \eqref{eq:three-signs}, see Remark~\ref{rmk:linear}. So we start from a general quadratic polynomial and derive coefficient conditions under which all three inequalities hold. These conditions are satisfied in sufficiently large dimensions and give the desired family of counterexamples.

We consider quadratic $O(n-1)$-invariant polynomials
\begin{equation}\label{eq:quadratic-polynomial}
P(a,b)=c_{0,0}+c_{1,0}a+c_{0,1}b+c_{2,0}a^2+c_{1,1}ab+c_{0,2}b^2, \qquad \Phi(x)=P(a,b).
\end{equation}
The parity assumption in Theorem~\ref{thm:criterion} is automatic. By $O(n-1)$ invariance, it is enough to evaluate $\mathcal{A}_\Phi$ at $(\sqrt{a},\sqrt{b},0,\ldots,0)$. At such a point,
\[
\mathrm{Hess}\, \Phi=
\begin{pmatrix}
2P_a+4aP_{aa}&4\sqrt{ab}\,P_{ab}&0\\
4\sqrt{ab}\,P_{ab}&2P_b+4bP_{bb}&0\\
0&0&2P_b I_{n-2}
\end{pmatrix}.
\]
Substituting the expression for $\mathrm{Hess}\,\Phi$ into \eqref{eq:linearized-data} gives
\[
\mathcal A_\Phi=
\begin{pmatrix}U&W\\W&V\end{pmatrix}\oplus T I_{n-2},
\qquad W=-4(n-2)c_{1,1}\sqrt{ab}, 
\]
where $U,V,T$ are affine on $\Delta_2$:
\[
U=U_\rho\rho+U_a a+U_b b, \qquad V=V_\rho\rho+V_a a+V_b b, \qquad T=T_\rho\rho+T_a a+T_b b.
\]
Their vertex values are
\begin{align*}
U_\rho&=-2(n-1)(c_{1,0}+c_{0,1}),\\
U_a&=-2(n-1)(c_{1,0}+c_{0,1}+6c_{2,0}+c_{1,1}),\\
U_b&=-2(n-1)(c_{1,0}+c_{0,1}+c_{1,1})-4(n+1)c_{0,2},\\
V_\rho&=-2c_{1,0}-2(2n-3)c_{0,1},\\
V_a&=V_\rho-12c_{2,0}-2(2n-3)c_{1,1},\\
V_b&=V_\rho-2c_{1,1}-4(4n-5)c_{0,2},\\
T_\rho&=V_\rho,\qquad T_a=V_a, \qquad 
T_b= V_b + 8(n-2)c_{0,2}.
\end{align*}
Define
\begin{equation}\label{eq:Dab-def}
D_{ab}=U_aV_b+U_bV_a-16(n-2)^2c_{1,1}^2.
\end{equation}
The determinant of the nontrivial block has the expression
\begin{align*}
UV-W^2={}&U_\rho V_\rho\rho^2+U_aV_a a^2+U_bV_b b^2+D_{ab}ab\\
 &+(U_aV_\rho+U_\rho V_a)a\rho +(U_bV_\rho+U_\rho V_b)b\rho.
\end{align*}
Consequently, the following finite system is sufficient for $\mathcal{A}_\Phi>0$ on $\overline{\mathbb{B}^n}$:
\begin{equation}\label{eq:quadratic-Ricci-sufficient}
U_\rho,U_a,U_b,V_\rho,V_a,V_b,T_b,D_{ab}>0.
\end{equation}

On $\mathbb{S}^{n-1}$, where $a+b=1$, one has $\partial_\nu\Phi=2aP_a+2bP_b$. Hence
\begin{align}
q_\Phi
 &=-c_{0,0}+c_{1,0}a+c_{0,1}b+3c_{2,0}a^2+3c_{1,1}ab+3c_{0,2}b^2\notag\\
 &=-c_{0,0}(a+b)^2+c_{1,0}a(a+b)+c_{0,1}b(a+b)
   +3c_{2,0}a^2+3c_{1,1}ab+3c_{0,2}b^2\notag\\
 &=(Q_a-c_{0,0}) a^2+2(Q_m-c_{0,0})ab+(Q_b-c_{0,0}) b^2 \label{eq:q-expansion},
\end{align}
where
\begin{equation}\label{eq:Q-definitions}
\begin{aligned}
 Q_a=c_{1,0}+3c_{2,0},\quad
 Q_m=\frac{c_{1,0}+c_{0,1}+3c_{1,1}}2,\quad
 Q_b=c_{0,1}+3c_{0,2}.
\end{aligned}
\end{equation}
The inequalities $Q_a,Q_m,Q_b>c_{0,0}$ imply $q_\Phi>0$.

{For integers $i,j\ge 0$}, with $(z)_k=z(z+1)\cdots(z+k-1)$ denoting the
rising factorial, the beta integral yields
\begin{align*}
 &\frac1{\omega_n}\int_{\mathbb{B}^n}a^ib^j \dd x
 =\frac{n}{n+2i+2j}
\frac{\left(\frac12\right)_i\left(\frac{n-1}{2}\right)_j}
 {\left(\frac n2\right)_{i+j}},
 \\
 &\frac1{\omega_n}\int_{\mathbb{S}^{n-1}}a^ib^jx_1^2 \dd S
 =n\frac{\left(\frac12\right)_{i+1}\left(\frac{n-1}{2}\right)_j}
         {\left(\frac n2\right)_{i+j+1}}.
\end{align*}
Therefore,
\begin{equation*}\label{eq:ell-moment}
  \mathcal{L}_n(a^ib^j)={}\frac{n(n-2)}{n+2i+2j} \frac{(\frac12)_i(\frac{n-1}{2})_j}{(\frac n2)_{i+j}}-n(n-1)\frac{(\frac12)_{i+1}(\frac{n-1}{2})_j}{(\frac n2)_{i+j+1}}.
\end{equation*}
In particular,
\begin{align*}
\mathcal{L}_n(1)&=-1,&
\mathcal{L}_n(a)&=-\frac{2n-1}{n+2},&
\mathcal{L}_n(b)&=-\frac{n-1}{n+2},&\\
\mathcal{L}_n(a^2)&=-\frac{3(4n-3)}{(n+2)(n+4)},&
\mathcal{L}_n(ab)&=-\frac{(n-1)(2n-1)}{(n+2)(n+4)},&
\mathcal{L}_n(b^2)&=-\frac{n^2-1}{(n+2)(n+4)}.&
\end{align*}
The moment identities above yield
\begin{align*}
-\mathcal{L}_n(\Phi)=
&c_{0,0} +\frac{2n-1}{n+2}c_{1,0}+\frac{n-1}{n+2}c_{0,1} +\frac{3(4n-3)}{(n+2)(n+4)}c_{2,0}\\
&+\frac{(n-1)(2n-1)}{(n+2)(n+4)}c_{1,1} +\frac{n^2-1}{(n+2)(n+4)}c_{0,2}.
\end{align*}

We are now in a position to describe the feasible region explicitly. Define
\begin{align}
 S_n={}&-\frac{2n-1}{n+2}c_{1,0}-\frac{n-1}{n+2}c_{0,1} -\frac{3(4n-3)}{(n+2)(n+4)}c_{2,0}\notag\\
 &-\frac{(n-1)(2n-1)}{(n+2)(n+4)}c_{1,1} -\frac{n^2-1}{(n+2)(n+4)}c_{0,2},
 \label{eq:S-lower-bound}\\
 M_n={}&\min\left\{Q_a, \, Q_{m},\, Q_b\right\}.
 \label{eq:M-upper-bound}
\end{align}
 Then set
\begin{equation}\label{eq:quadratic-feasible-region}
\begin{aligned}
\mathfrak{F}_n=\Bigl\{
\mathbf c=& (c_{0,0}, c_{1,0}, c_{0,1}, c_{2,0}, c_{1,1}, c_{0,2}) \in\mathbb R^6 :\\
{}&\;U_\rho,U_a,U_b,V_\rho,V_a,V_b,T_b,D_{ab}>0, 
\; S_n<c_{0,0}<M_n \Bigr\}.
\end{aligned}
\end{equation}
By the definitions of $S_n$ and $M_n$, the inequality $c_{0,0}>S_n$ is equivalent to $\mathcal{L}_n(\Phi)<0$, while $c_{0,0}<M_n$ is equivalent to $Q_a,Q_m,Q_b>c_{0,0}$. Thus, every coefficient vector in $\mathfrak{F}_n$ satisfies all three hypotheses of Theorem~\ref{thm:criterion}. 

\begin{remark}\label{rmk:linear}
If $P=c_0+c_1a+c_2 b$ is linear, then the Ricci condition implies \[
c_1+c_2<0, \ \ \ c_1+(2n-3)c_2<0,
\]  
while the boundary condition implies $c_0<\min\{c_1,c_2\}$. Since 
\[
\mathcal{L}_n(a)=-\frac{2n-1}{n+2}, \ \ \mathcal{L}_n(b)=-\frac{n-1}{n+2},
\]
we have 
\[
\mathcal{L}_n(\Phi)=-c_0+c_1\mathcal{L}_n(a)+c_2\mathcal{L}_n(b).
\]
If $c_1\leq c_2$, then $c_1<0$, $c_2<-c_1$, and
\[
\mathcal{L}_n(\Phi)> -c_1+c_1\mathcal{L}_n(a)-c_1\mathcal{L}_n(b)>0.
\]
If $c_2\leq c_1$, then $c_2<0$, $c_1<-c_2$, and
\[
\mathcal{L}_n(\Phi)>-c_2 -c_2\mathcal{L}_n(a)+c_2\mathcal{L}_n(b)>0.
\]
Hence, the function $\Phi(x)=P(a,b)$ fails to satisfy the three sign conditions in \eqref{eq:three-signs}.
\end{remark}

We now construct an explicit quadratic polynomial satisfying the three conditions of Theorem~\ref{thm:criterion} for every $n\geq 100$. Since the feasible region $\mathfrak{F}_n$ is not a singleton, the choice of such a quadratic polynomial is not unique. Here, we present one relatively simple explicit solution by assuming $c_{i,j} = p_{i,j} + k_{i,j}/n$ and solving possible $p_{i,j}, k_{i,j}$ from $\mathfrak{F}_n$. 

We start with the simplest term $U_{\rho}$. {We seek coefficients for which $U_{\rho}$ converges to a positive finite limit as $n \to \infty$.} Thus, we choose $c_{1,0} + c_{0,1} = O(\frac{1}{n})$. Furthermore, to ensure $U_{\rho} > 0$, we need $c_{1, 0} + c_{0, 1}<0$.  
{We may assume that $c_{1,0}=1$ and 
\[
c_{0,1} = -1 - \frac{\alpha}{n}
\]
for some $\alpha>0$.} Then we have
\[
U_\rho=2\alpha\left(1-\frac1n\right)\to 2\alpha>0
\]
and
\[
 V_{\rho} = 4(n-2) + \frac{2(2n-3)\alpha}{n} > 0.
\]
Set
\[
c_{0, 2}=\beta>0.
\]
Since $T_b-V_b = 8(n-2)\beta > 0$, we know that $V_b > 0$ implies $T_b > 0$. On the other hand, to ensure that $U_b > 0$, we have
\[
c_{1,1} <-2\beta\frac{n+1}{n-1} + \frac{\alpha}{n} =-2\beta-\frac{4\beta}{n}-\frac{4\beta}{n(n-1)} + \frac{\alpha}{n}.
\]
For large $n$, we may choose
\[
c_{1,1} = -2\beta-\frac{4\beta}{n}.
\]
Thus,
\[
U_b= 2\alpha\left(1-\frac1n\right)-\frac{8\beta}{n}\to 2\alpha>0.
\]
Set
\[
c_{2, 0}=\gamma.
\]
Taking $n\to\infty$, we obtain
\[
\begin{aligned}
\frac{U_a}{n}\to 4\beta-12\gamma,\
\frac{V_\rho}{n}\to 4,\
\frac{V_a}{n}\to 4+8\beta,\
\frac{V_b}{n}\to 4-16\beta,\
\frac{T_b}{n}\to 4-8\beta,
\end{aligned}
\]
and 
\[
\begin{aligned}
\frac{D_{ab}}{n^2} \to 16\left[\beta(1-8\beta)-3\gamma(1-4\beta)\right].
\end{aligned}
\]
{Therefore, to ensure Ricci positivity conditions \eqref{eq:quadratic-Ricci-sufficient}, we choose} $0<\beta<\frac{1}{4}$ and
\begin{equation}\label{eq:gamma_upper}
\gamma<\frac{\beta(1-8\beta)}{3(1-4\beta)}.
\end{equation}
Taking $n\to\infty$ again, we obtain
\[
Q_a \to 1 + 3\gamma,\quad Q_m \to -3\beta,\quad Q_b \to -1 + 3\beta
\]
and
\[
S_n\to -1 + 3\beta.
\]
A direct calculation shows
\[
n(Q_b - S_n) \to  2(6\gamma+8\beta-\alpha-1).
\]
{To ensure that $Q_b-S_n>0$ for all sufficiently large $n$, it suffices to require}
\[
0 < \alpha < 6\gamma+8\beta-1.
\]
Thus, to guarantee the existence of $\alpha$, we require the following inequality
\begin{equation}\label{eq:gamma_lower}
\gamma>\frac{1-8\beta}{6}.
\end{equation}
Combining \eqref{eq:gamma_upper} and \eqref{eq:gamma_lower}, we have
\[
\frac{\beta(1-8\beta)}{3(1-4\beta)}-\frac{1-8\beta}{6}
=\frac{(8\beta-1)(1-6\beta)}{6(1-4\beta)}>0.
\]
Therefore, to guarantee the existence of $\gamma$, the constant $\beta$ must satisfy 
\[
\frac{1}{8}<\beta<\frac{1}{6}.
\]
In this case, $Q_b<Q_m<Q_a$ {for all sufficiently large $n$}. Thus, $S_n<M_n$ follows from $S_n<Q_b$.

In conclusion, if {the parameters} $\alpha, \beta, \gamma$ satisfy
\begin{equation}\label{eq:condition_of_coefficients}
0<\alpha<6\gamma+8\beta-1,\quad
\frac18<\beta<\frac16,\quad
\frac{1-8\beta}{6}<\gamma<\frac{\beta(1-8\beta)}{3(1-4\beta)},
\end{equation}
then for each $S_n<c_{0,0}<M_n$,
\[
\mathbf{c}= \left(c_{0,0},\ 1,\ -1-\frac{\alpha}{n},\ \gamma,\ -2\beta-\frac{4\beta}{n},\ \beta\right) \in \mathfrak{F}_n
\]
for all sufficiently large $n$. Therefore, for all sufficiently large $n$, the feasible region $\mathfrak{F}_n$ is nonempty. Moreover, if we choose
\begin{equation*}
\alpha = \frac{1}{60},\quad
\beta = \frac{3}{20},\quad
\gamma = -\frac{1}{36},
\end{equation*}
then we can establish the following proposition.

\begin{proposition}\label{prop:quadratic}
For every $n \geq 100$, the set $\mathfrak{F}_n$ is nonempty.
\end{proposition}
\begin{proof}
For any $n \geq 100$, set
\[
c_{1,0}=1,\qquad
c_{0,1}=-1-\frac{1}{60n},\qquad
c_{2,0}=-\frac{1}{36},\qquad
c_{1,1}=-\frac{3}{10}-\frac{3}{5n},\qquad
c_{0,2}=\frac{3}{20}.
\]
Recall that by construction, $U_{\rho}$ and $V_{\rho}$ are always positive, and $V_b > 0$ implies that $T_b > 0$. Moreover, we have $V_a = V_{\rho} - 12c_{2, 0}-2(2n-3)c_{1, 1} > 0$. Substituting the coefficients into the remaining vertex formulas, we get
\[
\begin{aligned}
U_a&=\frac{(n-1)(28n+37)}{30n} > \frac{14n}{15} > 0, \quad 
&U_b&=\frac{n-37}{30n}>0, \\
V_b&=\frac{48n^2-130n+33}{30n} > \frac{31n}{20} > 0.
\end{aligned}
\]
The remaining Ricci coefficient $D_{ab}$, defined by \eqref{eq:Dab-def}, satisfies
\begin{align*}
D_{ab}&>U_aV_b-16(n-2)^2c_{1,1}^2 >\frac{14n}{15}\cdot \frac{31n}{20} - 16n^2\left(\frac{n-2}{n}c_{1,1}\right)^2 \\
&> \frac{14n}{15}\cdot \frac{31n}{20} - 16n^2\left(\frac{3}{10}\right)^2>0
\end{align*}
for $n\geq100$. Thus, the Ricci positivity conditions in \eqref{eq:quadratic-Ricci-sufficient} are satisfied.

It remains to choose the constant coefficient $c_{0,0}$. It is sufficient to show that $S_n < M_n$. By \eqref{eq:M-upper-bound} and \eqref{eq:S-lower-bound},
\[
M_n
=\min\left\{\frac{11}{12},\,-\frac{9}{20}-\frac{109}{120n},\, -\frac{11}{20}-\frac{1}{60n}\right\}
=-\frac{11}{20}-\frac{1}{60n},
\]
and
\[
S_n =-\frac{33n^3+201n^2+93n-32}{60n(n+2)(n+4)}.
\]
For $n\geq100$,
\[
M_n-S_n
=\frac{2n^2-177n-40}{60n(n+2)(n+4)}>0.
\]
Then we can choose
\[
c_{0,0}=\frac{S_n+M_n}{2}.
\]
Hence $S_n<c_{0,0}<M_n$. Thus, $\mathbf{c}=(c_{0,0},c_{1,0},c_{0,1},c_{2,0},c_{1,1},c_{0,2}) \in\mathfrak{F}_n$,
and $\mathfrak{F}_n$ is nonempty. This completes the proof.
\end{proof}

At the end of this section, we complete the proof of Theorem~\ref{thm:polynomial}.
\begin{proof}[Proof of Theorem~\ref{thm:polynomial}]
The conclusion follows from Theorem~\ref{thm:criterion} and Proposition~\ref{prop:quadratic}.
\end{proof}

\section*{Acknowledgments}
The authors thank Professor Bobo Hua for his guidance and constant support. The second author is supported by NSFC (Nos.~12101125 and 12371052) and the Fujian Alliance of Mathematics (No.~2024SXLMMS01). The third author is partially supported by the National Key R\&D Program of China (No.~2025YFA1017500).
    
\section*{Declarations}
	
\noindent\textbf{Conflict of Interest}\ \  
The authors declare that they have no conflict of interest.
	
\medskip
\noindent\textbf{Data Availability}\ \  
No datasets were generated or analyzed in this study. 

\medskip
\noindent\textbf{AI assistance statement}\ \  The authors used GPT-5.6 Pro to assist with symbolic checking and manuscript editing. {All mathematical statements and proofs were developed by the authors}, who take full responsibility for the content and accuracy of the article.	

\bibliographystyle{plain}
\bibliography{counterexample}

\end{document}